\documentclass{amsart}

\usepackage{amsmath}

\usepackage{amsfonts}

\usepackage{hyperref}

\usepackage{latexsym}

\usepackage{amssymb}

\usepackage[utf8]{inputenc} 

\usepackage{amscd}

\usepackage{geometry}
\numberwithin{equation}{section}

\newtheorem{theorem}{Theorem}[section]
\newtheorem{lemma}[theorem]{Lemma}
\newtheorem{corollary}[theorem]{Corollary}
\newtheorem{proposition}[theorem]{Proposition}
\newtheorem{definition}[theorem]{Definition}

\newtheorem{theorem*}{Theorem}
\newtheorem{remark*}{Remark}

\newcommand{\ma}[1]{\ensuremath{\mathbb{#1}}}
 
\font\bb=msbm7 at 10 pt

\def \Z {\hbox{\bb Z}}

\def \N {\hbox{\bb N}}

\def \F {\hbox{\bb F}}

\def \Tr {\mbox{\rm{Tr}}}

\def \I {\mbox{\bf I}}

\newcommand{\NP}{\ensuremath{\mbox{\rm{NP}}}}
\newcommand{\GNP}{\ensuremath{\mbox{\rm{GNP}}}}
\newcommand{\Ima}{\ensuremath{\mbox{\rm{Im }}}}

\author{R\'egis Blache}
\address{\'Equipe LAMIA,
\'ESP\'E de Guadeloupe}
\email{rblache@espe-guadeloupe.fr}

\title[First vertices]{First vertices for hyperelliptic curves in characteristic two}

\begin{document}

\begin{abstract}
We study the Newton polygons of numerators of the zeta functions of $2$-rank $0$ hyperelliptic curves in characteristic $2$. We determine their first generic vertex, and their first vertex in some other non generic cases.
\end{abstract}

\subjclass[2000]{}
\keywords{}

\maketitle

\section*{Introduction}

In this paper, we consider hyperelliptic curves in characteristic two having $2$-rank $0$. Precisely, we try to determine the first vertex of the Newton polygon of the numerator of their zeta function in some cases. The stratification by the Newton polygons of the moduli space of principally polarized abelian varieties has been studied \cite{oo}, and this is a way to study the image of hyperelliptic curves under the Torelli morphism in this moduli space.

These questions have already drawn some attention. In \cite{vv}, van der Geer and van der Vlugt study some families of supersingular (i.e. having the highest possible Newton polygon) hyperelliptic curves; then they use these families in \cite{vv2} to show that there exist supersingular curves of any genus in characteristic two. On the other hand, in \cite{sz}, Scholten and Zhu give a lower bound for the first slope of the Newton polygon of such an hyperelliptic curve, and sufficient conditions for a given curve to reach this bound. The same authors give all possible first slopes for the Newton polygons of $2$-rank $0$ hyperelliptic curves in characteristic two when the genus is at most $8$, see \cite{sz1}.

Recall from \cite[Proposition 4.1]{sz} that a genus $g$ hyperelliptic curve having $2$-rank $0$ defined over the finite field $k=\F_q$ admits an equation of the form
$$y^2+y=f(x)$$
where $f(x):=\sum_{i=0}^{g}c_{2i+1}x^{2i+1}$ is a polynomial of degree $2g+1$. We shall denote this curve by $C_f$ in the following. Its zeta function is rational, and we denote by $\NP_q(C_f)$ the Newton polygon of its numerator $L(C_f,T)$ with respect to the $q$-adic valuation normalized by $v_q(q)=1$. This is a convex polygon with end points $(0,0)$ and $(2g,g)$, positive slopes since the $2$-rank is zero, and break points having integer coordinates.

If $\psi$ denotes a non-trivial additive character of $\F_q$, one can associate to $f$ the following family of exponential sums, and the associated $L$-function
$$S_m(f):=\sum_{x\in \ma{F}_{q^m}}\psi(\Tr_{\ma{F}_{q^m}/\ma{F}_q}(f(x))),~ L(f,T)=\exp\left(\sum_{m\geq 1} S_m(f)\frac{T^m}{m}\right)$$
We have $L(C_f,T)=L(f,T)$, and the congruence given in \cite{cong} applies to this last function. Along this paper, we collect the information necessited to write down this congruence explicitely in some cases. Once this has been done, the determination of the first vertex follows from some simple semi-algebra.

Our first result precises \cite[Theorem 1.1]{sz}.
\begin{theorem*}
\label{fv}
Assume $g\geq 3$, and set $n:=\lfloor\log_2(2g+2)\rfloor$. 
\begin{itemize}
	\item[(i)] When $2^n-1\leq 2g+1<2^{n+1}-3$, the first vertex of $\NP_q(C_f)$ is $(n,1)$ if, and only if we have $c_{2^n-1}\neq 0$;
	 \item[(ii)] Assume $2g+1=2^{n+1}-3$; 
	 
	 \begin{itemize}
	\item[(a)] the first vertex of $\NP_q(C_f)$ is $(2n,2)$ if, and only if we have $c_{3\cdot 2^{n-1}-1}\neq 0$;
	 \item[(b)] when $c_{3\cdot 2^{n-1}-1}= 0$, the first vertex of $\NP_q(C_f)$ is $(n,1)$ if, and only if we have $c_{2^n-1}\neq 0$.
\end{itemize}
\end{itemize}
\end{theorem*}

Note that assertions (i) and (iia) give the first vertex of the generic Newton polygon associated to the family of genus $g$ and $2$-rank $0$ hyperelliptic curves by Grothendieck's specialization theorem. In the case $2g+1=2^n-1$, we must have $c_{2^n-1}\neq 0$ and the first vertex is $(n,1)$ for all curves $C_f$. If moreover $n\geq 3$, we get \cite[Theorem 1.2]{sz}: there does not exist any supersingular elliptic curve of genus $g=2^{n-1}-1$ in characteristic two.

Assertion (iib) is a first step towards the general case. We precise this in the next result; actually, when the above coefficients vanish, we determine the first vertex of the generic Newton polygon for the resulting family of curves.

\begin{theorem*}
\label{fv2}
Notations are as above; assume that $g$ is large enough
\begin{itemize}
\item[(i)] when $2^n-1<2g+1<3\cdot2^{n-1}-1$, and $c_{2^n-1}=0$, we have the following possible first vertices for $\NP_q(C_f)$
\begin{itemize}
\item[(a)] if $2^n-1<2g+1<5\cdot2^{n-2}-1$, it is $(2n-2,2)$ if, and only if $c_{2^n-3}c_{3\cdot2^{n-2}-1}\neq 0$;
\item[(b)] if $5\cdot2^{n-2}-1\leq 2g+1<3\cdot2^{n-1}-5$, it is $(2n-2,2)$ if, and only if $c_{2^n-3}^{2^{n-2}}c_{3\cdot2^{n-2}-1}+c_{2^n-5}^{2^{n-2}}c_{5\cdot2^{n-2}-1}\neq 0$;
\item[(c)] if $2g+1=3\cdot2^{n-1}-5$, it is $(3n-3,3)$ if, and only if $c_{2^n-3}c_{3\cdot2^{n-1}-5}c_{5\cdot2^{n-2}-1}\neq 0$;
\item[(d)] if $2g+1=3\cdot2^{n-1}-3$, it is $(3n-3,3)$ if, and only if $c_{5\cdot2^{n-2}-1}(c_{2^n-3}c_{3\cdot2^{n-1}-5}+c_{2^n-5}c_{3\cdot2^{n-1}-3})\neq 0$;
\end{itemize}
\item[(ii)] when $3\cdot2^{n-1}-1\leq 2g+1<2^{n+1}-7$, and $c_{2^n-1}=0$, the first vertex of $\NP_q(C_f)$ is $(2n-1,2)$ if, and only if $c_{2^{n}-3}c_{3\cdot2^{n-1}-1}\neq 0$. Else the first slope is at least $\frac{1}{n-1}$.
\item[(iii)] when $2g+1=2^{n+1}-7$, and $c_{2^n-1}=0$, the first vertex of $\NP_q(C_f)$ is $(2n-1,2)$ if, and only if
$$c_{2^{n+1}-7}^{2^{n-2}}c_{7\cdot2^{n-2}-1}+
c_{2^{n}-3}^{2^{n-1}}c_{3\cdot2^{n-1}-1}\neq 0$$
\item[(iv)] when $2g+1=2^{n+1}-5$, and $c_{2^n-1}=0$, the first vertex of $\NP_q(C_f)$ is $(2n-1,2)$ if, and only if
$$c_{2^{n+1}-5}^{2^{n-2}}c_{5\cdot2^{n-2}-1}+
c_{2^{n+1}-7}^{2^{n-2}}c_{7\cdot2^{n-2}-1}+
c_{2^{n}-3}^{2^{n-1}}c_{3\cdot2^{n-1}-1}\neq 0$$
\item[(v)] when $2g+1=2^{n+1}-3$, and $c_{2^n-1}=c_{3\cdot2^{n-1}-1}=0$, the first vertex of $\NP_q(C_f)$ is $(2n-1,2)$ if, and only if
$$c_{2^{n+1}-3}^{2^{n-2}}c_{3\cdot2^{n-2}-1}+
c_{2^{n+1}-5}^{2^{n-2}}c_{5\cdot2^{n-2}-1}+
c_{2^{n+1}-7}^{2^{n-2}}c_{7\cdot2^{n-2}-1}\neq 0$$
\end{itemize} 
\end{theorem*}

Note that this result improves the bounds given in \cite[Theorem 1.3]{sz}.

\begin{remark*}
The first slopes of the segments give a lower bound (which does not depend on $m$) on the $q^m$-adic valuation of the exponential sum $S_m(f)$, when $f$ satisfies the corresponding conditions.
\end{remark*}

Let us briefly present the structure of the paper: in Section \ref{s1}, we recall certain modular equations defined in \cite{mkcs} and some of the invariants associated to their solutions (see \cite{dens}), in particular their supports. Then we give properties of these supports, already studied in \cite{blabla}, but specialized to the case $p=2$. These results allow us to determine the solutions of low density in Section \ref{s2}. With this at hand, we can write explicitely the congruence given in \cite{cong} in some cases; this is done in Section \ref{s3}, and used to show the two theorems above.

\section{The supports of solutions of modular equations}
\label{s1}

In this short section, we rewrite the properties of supports of solutions of the modular equation given in \cite[Section 1]{blabla} in the case $p=2$. In the following, $D$ denotes a non empty subset of the set of positive integers.

For any $\ell\geq 1$, we define the finite set 
$E_{D,p}(\ell)\subset \{0,\ldots,p^\ell-1\}^{|D|}$ as the set of solutions $U=(u_d)_{d\in D}$ of the following system (see \cite{mkcs})
\begin{equation}
\label{mod}
\left\{\begin{array}{rcl}
\sum_D du_d & \equiv & 0 \mod p^\ell-1 \\
\sum_D du_d & > & 0 \\
\end{array}
\right.
\end{equation}

We denote by $s_p(n)$ the $p$-weight of the integer $n$, i.e. the sum of its base $p$ digits. We define the \emph{weight} of a solution as 
$s_p(U):=\sum_D s_p(u_d)$, its \emph{length} as $\ell(U):=\ell$, and its \emph{density} as $\delta(U):=\frac{s_p(U)}{(p-1)\ell(U)}$.

We set $\sigma_{D,p}(\ell):=\min\{s_p(U),~U\in E_{D,p}(\ell)\}$. In \cite{dens}, we have shown that the infimum
$$\inf_{\ell\geq 1} \left\{ \frac{\sigma_{D,p}(\ell)}{\ell(p-1)}\right\}$$
is actually a minimum $\delta_{D,p}$, the \emph{$p$-density} of the set $D$.

\begin{definition}
A solution $U\in E_{D,p}(\ell)$ is \emph{minimal} when we have $\delta(U)=\delta_{D,p}$.
\end{definition}

We define the \emph{shift} as the map $\delta$ from $\{0,\ldots,p^\ell-1\}$ to itself leaving $p^\ell-1$ fixed, and sending any other $i$ to the 
remainder of $pi$ modulo $p^\ell-1$ (note that this map shifts the base $p$ digits). We extend it coordinatewise to the set 
$\{0,\ldots,p^\ell-1\}^{|D|}$; then it leaves the subset $E_{D,p}(\ell)$ stable. As a consequence, all integers $\sum_D d\delta^k(u_d)$, $0\leq k\leq \ell-1$, are positive multiples of $p^\ell-1$. 

\begin{definition}
The \emph{support} of the solution $U$ is the map $\varphi_U$ from $\Z/\ell\Z$ to $\N_{>0}$ 
defined by 
$$\varphi_U(k):=\frac{1}{p^\ell -1} \sum_D d\delta^k(u_d)$$
A solution $U$ is \emph{irreducible} when the map $\varphi_U$ is an injection.
\end{definition}

For any $d\in D$ we write the base $p$ expansion $u_d=\sum_{r=0}^{\ell-1} p^ru_{dr}$; note that we have $s_p(U)=\sum_D\sum_{r=0}^{\ell-1}u_{dr}$. Recall from \cite[Lemma 1.2 (ii)]{dens} that for any 
$0\leq r\leq \ell-1$, we have the equalities

\begin{equation}
\label{saut}
\sum_D du_{dr}= p\varphi_U(\ell-r-1)-\varphi_U(\ell-r)
\end{equation}

Let us define a certain type of maps as in \cite[Section 1.2]{blabla}.

\begin{definition}
Let $\ell\geq s$ denote two integers, and $\varphi : \Z/\ell\Z\rightarrow \N_{>0}$ any map 

\begin{itemize}
	\item[(i)] We say that $\varphi$ is a \emph{support map of length $\ell$ with $s$ jumps} if we have $\varphi(i+1)=p\varphi(i)$ except for exactly 
$s$ pairwise distinct values $i_1,\ldots,i_s\in \Z/\ell\Z$, for which we have $\varphi(i+1)<p\varphi(i)$.
\item[(ii)] We say that $\varphi$ is \emph{irreducible} when $\varphi$ is an injection.
\end{itemize}

\end{definition}

We give the link between the supports of solutions of modular equations, and the maps we have just defined. The following is the special case $p=2$ of \cite[Proposition 1.11]{blabla}

\begin{lemma}
\label{supp}
Let $U$ be a solution of the system (\ref{mod}) associated to $D$ and $p=2$, with weight $w$ and length $\ell$. Then
\begin{itemize}
	\item[(i)] its support $\varphi_U$ is a support map of length $\ell$, with at most $w$ jumps; moreover it is irreducible if, and only if
	the solution $U$ is;
	\item[(ii)] if the support $\varphi_U$ has $s$ jumps, then we have the following inequality
	$$  \max \varphi_U\leq  (w-s+1)\max D $$
	\item[(iii)] if it has exactly $w$ jumps, then all are elements of $D$. Moreover, the solution $U$ is completely determined by its support in this case, and we have $u_{dr}\in \{0,1\}$ for any $d,r$.
\end{itemize}
\end{lemma}

We end this section with the case $p=2$ of \cite[Proposition 1.9]{blabla}; it will be very useful in the next section, in order to give an upper bound for the lengths of solutions having low density.

\begin{proposition}
\label{bou2}
Let $\varphi$ denote an irreducible support map of length $\ell$ with $t$ jumps, $t\leq s<\ell$. Write $\ell=qs+r$ with $1\leq r\leq s$ and 
$q\geq 1$; then we have 
$$|\varphi|=\sum_{i=0}^{\ell-1} \varphi(i) \geq \frac{1}{2}s(s+1) + (2^{q-1}-1)\frac{1}{2}\left(3s^2+s\right)+2^{q-2}r(2s+r+1)$$
\end{proposition}

\begin{proof}
First assume $t=s$. As in \cite[Lemma 1.7]{blabla}, we define the sequence $(c_i)$ by $c_i=i$ for $1\leq i< 2s$, and $c_i=2c_{i-s}$ for $i\geq 2s$. As in the proof of \cite[Proposition 1.9]{blabla}, the irreducibility of $\varphi$ guarantees $|\varphi|\geq \sum_{i=1}^\ell c_i$. Now we have

$$ \sum_{i=1}^s c_i=\frac{1}{2}s(s+1),~ \sum_{i=ks+1}^{(k+1)s} c_i=2^{k-1}\sum_{i=s+1}^{2s} c_i=2^{t-1}\frac{1}{2}\left(3s^2+s\right)$$
for any positive integer $k$. Finally we have 
$$\sum_{i=qs+1}^{qs+r} c_i=2^{q-1}\sum_{i=s+1}^{s+r} c_i=2^{q-1}\frac{1}{2}\left((s+r+1)(s+r)-s(s+1)\right),$$
and this gives the result.

If $t<s$, the bound $|\varphi|\geq \sum_{i=1}^\ell c_i$ remains valid (see \cite[proof of Lemma 1.10]{blabla}), and we conclude as above.
\end{proof}

\section{Solutions of the modular equation having low density}
\label{s2}

We fix an integer $n\geq 3$ in the following, and we set $D:=\{1\leq i\leq 2^{n+1}-3,~i\equiv 1\mod 2\}$. We determine all irreducible solutions of the modular equations associated to $D$ and $p=2$ and having density in the interval $\left[\frac{1}{n},\frac{1}{n-1}\right]$.

All along this section, $U$ denotes such a solution, with length 
$\ell$ and weight $w$. We must have $\frac{1}{n}\leq\frac{w}{\ell}\leq\frac{1}{n-1}$, and $(n-1)w\leq\ell\leq nw$. Thus we have $\ell=(n-1)w+r$ for some $0\leq r\leq w$.

\subsection{Properties of the support}

We begin with a bound for the weight of such a solution $U$

\begin{lemma}
\label{weight}
Let $U$ be as above; then the couple $(w,r)$ lies in the following set
$$\left\{ (i,0),~ 1\leq i\leq 5,~ (1,1),(2,2),(2,1),(3,1),(4,1) \right\}$$
\end{lemma}

\begin{proof}
We apply Proposition \ref{bou2} to the support of $U$ (it has at most $w$ jumps from Lemma \ref{supp} (i)), and we use the bound from \cite[Lemma 1.2]{dens}, to obtain
\begin{equation}
\label{bound22}
\frac{1}{2}w(w+1) + (2^{n-2}-1)\frac{1}{2}\left(3w^2+w\right)+2^{n-2}\frac{1}{2}r(r+2w+1)\leq w(2^{n+1}-3)
\end{equation}
Taking $r=0$, and symplifying by $w$, we get $(3\cdot 2^{n-2}-2)w\leq 15\cdot 2^{n-2}-6$, and $w\leq 5$ 
as long as $n\geq 3$. Thus $(w,r)$ must lie in $\{(i,j),~ 1\leq i\leq 5,~ 0\leq j\leq i\}$. Checking by hand the inequality (\ref{bound22}), 
the only possible cases are the $(i,0)$, $1\leq i\leq 5$, and $(1,1),(2,2),(2,1),(3,1),(3,2),(4,1)$. 

With a little more work, we can also exclude the case $(w,r)=(3,2)$. In this case, we get an equality in (\ref{bound22}); as a consequence, we must 
have equalities in both Proposition \ref{bou2} and \cite[Lemma 1.2]{dens}. Equality in \cite[Lemma 1.2]{dens} implies that $u_d=0$ except for 
$d=2^{n+1}-3$; thus we are looking for a solution of the form $u\cdot(2^{n+1}-3)\equiv 0 \mod 2^{3n-1}-1$ and $\sigma_2(u)=3$. We have that 
$\gcd(2^{n+1}-3,2^{3n-1}-1)=11$ when $n\equiv 7 \mod 10$, and $1$ else. In the first case, we have $11 u\equiv 0 \mod 2^{3n-1}-1$, and 
$\sigma_2(u)\geq n=\left\lceil\frac{3n-1}{\sigma_2(11)}\right\rceil$, a contradiction; in the second, we get $u\equiv 0 \mod 2^{3n-1}-1$ and 
$\sigma_2(u)\geq 3n-1$ from \cite[Proposition 11 (iv)]{mkcs}.

\end{proof}

Now we have a bound for the weight of a low density solution, we derive results about its support in the next lemmas.

\begin{lemma}
\label{jumps}
Let $U$ be as above; if $n$ is large enough, the support $\varphi_U$ has exactly $w$ jumps.
\end{lemma}

\begin{proof}
We assume the support has $s$ jumps, and we write it (up to shift)
$$n_1,\ldots,2^{\ell_1-1}n_1,\ldots,n_s,\ldots,2^{\ell_s-1}n_s$$
From the inequality in \cite[Lemma 1.2]{dens}, we have 
$$\sum_{i=1}^s n_i(2^{\ell_i}-1)\leq w(2^{n+1}-3)\leq 5(2^{n+1}-3)<2^{n+4}-1$$
and we get $\ell_i\leq n+3$ for all $i$. As a consequence, we get the following inequality for the lengths $(n-1)w\leq \ell=\sum \ell_i\leq s(n+3)$. If we have $s\leq w-1$, the inequality becomes $n\leq 4w-3$, and this is the desired result.
\end{proof}

\begin{lemma}
\label{cases}
Let $U$ denote a solution of length $\ell$ and weight $w$ for the modular equation. Assume that its support has $w$ jumps, and write it as above. Assume moreover that $n_k>2^u$ for some positive integer $u$; then at least one of the following assertions hold
\begin{itemize}
\item[(i)] $\ell_k\leq n-u$;
\item[(ii)] $\ell_k+\ell_{k+1}\leq n+1$
\end{itemize}
\end{lemma}

\begin{proof}
First note that since the support has $w$ jumps, we have the bound $2^{\ell_k-1}n_k\leq 2^{n+1}-3$ from Lemma \ref{supp} (ii). As a consequence, we have $\ell_k\leq n-u+1$.

Assume we have $\ell_k= n-u+1$ (i.e. assertion (i) does not hold). We must have $2^{\ell_k}n_k-n_{k+1}\in D$ since the support has $w$ jumps; as $n_k\geq 2^u+1$, we get the bound $n_{k+1}\geq 2^{n-u+1}+3$. The inequality $2^{\ell_{k+1}-1}n_{k+1}\leq 2^{n+1}-3$ guarantees $\ell_{k+1}\leq u$.

\end{proof}

\begin{lemma}
\label{length}
Let $U$ be as above, with weight $w\geq 2$ and support as above; assume moreover that $n$ is large enough. 
\begin{itemize}
\item[(i)] We have $\ell_k\leq n+1$ for all $1\leq k\leq w$
\item[(ii)] all $n_k$ are odd integers
\item[(iii)] If we have $n_k>2^u$ for some positive integer $u$, then we have the inequality $\ell_k\leq n-u$.
\item[(iv)] If we have $\ell_i\geq n$ for some $i$, then $n_i=1$, and all other lengths satisfy $\ell_j\leq n-1$. Moreover, there exists at most one $j$ such that $\ell_j=n-1$, and in this case $n_j=3$.
\end{itemize}
\end{lemma}

\begin{proof}
The first assertion follows from Lemma \ref{supp} (ii); since the support has $w$ jumps, we must have $n_k2^{\ell_k-1}\leq 2^{n+1}-3$ for all $k$. The second comes from assertion (iii) of the same lemma: the jumps are the $2^{\ell_k}n_k-n_{k+1}$, and they are in $D$; since $\ell_k\geq 1$ and all integers in $D$ are odd, we get the result.

Assume we have $n_k>2^u$ and $\ell_k\geq n-u+1$ for some $k$; from Lemma \ref{cases} (we know that the support of $U$ has $w$ jumps from Lemma \ref{jumps}), we must have $\ell_k+\ell_{k+1}\leq n+1$. As in the last proof, we get the inequality $(n-1)w\leq \ell=\sum \ell_i\leq (w-2)(n+3)+n+1$, from which we deduce $n\leq 4w-5$, and the third assertion.

Assume $\ell_i\geq n$; the contraposition of (iii), with $u=1$, ensures us we have $n_i\leq 2$, and $n_i=1$ since $n_i$ is odd. If moreover $\ell_j=n-1$, we get $n_j\leq 4$ in the same way, and $n_j=3$ since $n_j$ is odd and cannot be equal to $n_i$ from the irreducibility of $U$.
\end{proof}

\subsection{Solutions having low density}

In the following, we assume $n$ is large enough in order to apply the preceding results.

From Lemma \ref{weight}, we know that the possible densities of $U$ are $\frac{1}{n}<\frac{2}{2n-1}<\frac{3}{3n-2}<\frac{4}{4n-3}<\frac{1}{n-1}$. 

We shall treat these possible densities in increasing order. Let $U$ denote a solution having one of these densities; recall from Lemma \ref{jumps} that its support has $w$ jumps, and therefore can be written (up to shift)
$$n_1,\ldots,2^{\ell_1-1}n_1,\ldots,n_w,\ldots,2^{\ell_w-1}n_w$$

\subsubsection{Solutions having density $\frac{1}{n}$}

Let $U$ denote such a solution; in order to have $\delta(U)=\frac{1}{n}$, we must have $w=r$ in Lemma \ref{weight}; thus $w\in\{1,2\}$.

If we have $w=1$, then the support of $U$ is a geometric sequence of common ratio $2$, with initial term $n_1$, and inequality \ref{bound22} gives $(2^n-1)n_1\leq 2^{n+1}-3$, i.e. $n_1=1$; from Lemma \ref{supp} (iii), we get the solution
$$1\cdot(2^n-1)$$

When $w=2$, the support of $U$ consists of two geometric sequences from Lemma \ref{jumps}; we have $\ell_1+\ell_2=2n$, and $\ell_1:=\max\{\ell_i\}\geq n$. From Lemma \ref{length} (iv), we must have $\ell_1=n+1$, $\ell_2=n-1$, $n_1=1$ and $n_2=3$. From Lemma \ref{supp} (iii), we get the solution
$$2^{n-1}(2^{n+1}-3)+1\cdot(3\cdot2^n-1)=2^{2n}-1$$

\subsubsection{Solutions having density $\frac{2}{2n-1}$} Such a solution must have weight $2$ and length $2n-1$. If we assume $\ell_1>\ell_2$, we must have $\ell_1\geq n$.

If $\ell_1=n$, then $\ell_2=n-1$, and we have $n_1=1$ and $n_2=3$ from Lemma \ref{length} (iv), this gives the solution 
$$2^{n-1}\cdot(2^{n}-3)+1\cdot(3\cdot 2^{n-1}-1)=2^{2n-1}-1$$
From Lemma \ref{length} (i), the only other possible lengths are $\ell_1=n+1$ and $\ell_2=n-2$. Here again we get $n_1=1$, and assertions (ii) and (iii) of the same lemma give $n_2\in\{3,5,7\}$; up to shift we get the solutions 
$$2^{n-2}\cdot(2^{n+1}-i)+1\cdot(2^{n-2}i-1)=2^{2n-1}-1,~i\in \{3,5,7\}$$

\subsubsection{Solutions having density $\frac{3}{3n-2}$} Such a solution must have weight $3$ and length $3n-2$. If we assume $\ell_1=\max\{\ell_i\}$, we must have $\ell_1\geq n$, $n_1=1$ and $\ell_2,\ell_3\leq n-1$ from Lemma \ref{length} (iv). From the same assertion, we cannot have $\ell_2=\ell_3=n-1$; as a consequence, we get $\ell_1=n+1$, and $\{\ell_2,\ell_3\}=\{n-1,n-2\}$. Moreover the initial term $n_i$ corresponding to the length $\ell_i=n-1$ must be $3$, and the other one must be at most $8$, thus $5$ or $7$. We get the following solutions up to shift, for $i\in\{5,7\}$

$$\left\{
\begin{array}{l}
2^{2n-3}\cdot(2^{n+1}-3)+2^{n-2}\cdot(3\cdot2^{n-1}-i)+1\cdot(i\cdot 2^{n-2}-1)=2^{3n-2}-1 \\
2^{2n-3}\cdot(2^{n+1}-i)+2^{n-1}\cdot(i\cdot2^{n-2}-3)+1\cdot(3\cdot 2^{n-1}-1)=2^{3n-2}-1 \\
\end{array}
\right.$$

\subsubsection{Solutions having density $\frac{4}{4n-3}$} Such a solution must have weight $4$ and length $4n-3$. If we assume $\ell_1=\max\{\ell_i\}$, we must have $n\leq\ell_1\leq n+1$, $n_1=1$ and $\ell_2,\ell_3,\ell_4\leq n-1$ with at most one being equal to $n-1$ from Lemma \ref{length} (iv). We deduce that $\ell=4n-3=\sum \ell_i\leq n+1+n-1+2(n-2)$, a contradiction. There does not exist any solution having density $\frac{4}{4n-3}$.

\subsubsection{Solutions having density $\frac{1}{n-1}$} Such a solution has weight $w$ and length $(n-1)w$ for some $1\leq w\leq 5$. 

When $w=1$, the support is a geometric sequence of length $n-1$; from \cite[Lemma 1.2]{dens}, we get the inequality $n_1(2^{n-1}-1)\leq 2^{n+1}-3$, and $n_1\leq 3$. This gives the solutions 
\begin{equation}
\label{n-11}
\left\{ \begin{array}{l}
1\cdot(2^{n-1}-1)\\
1\cdot(3\cdot2^{n-1}-3)\\
\end{array}\right.
\end{equation}

When $w=2$, we have $\ell_1+\ell_2=2n-2$. Assume $\ell_1\geq \ell_2$; we have three possibilities. If $\ell_1=n+1$, $\ell_2=n-3$, we must have $n_1=1$ and $n_2\leq 16$ from Lemma \ref{length}. When $\ell_1=n$, $\ell_2=n-2$, we must have $n_1=1$ and $n_2\leq 8$. Finally, if $\ell_1=\ell_2=n-1$, we must have $\{n_1,n_2\}=\{1,3\}$. Summarizing, we get the solutions
\begin{equation}
\label{n-12}
\left\{ \begin{array}{ll}
2^{n-3}(2^{n+1}-i)+1\cdot(i\cdot2^{n-3}-1),& i\in \{3,5,\ldots,15\}\\
2^{n-2}(2^{n}-i)+1\cdot(i\cdot2^{n-2}-1),& i\in \{3,5,7\}\\
2^{n-1}(2^{n-1}-3)+1\cdot(3\cdot2^{n-1}-1)\\
\end{array}\right.
\end{equation}

For $w=3$, if we set $\ell_1=\max\{\ell_i\}$, we cannot have $\ell_1=n-1$: in this case all $\ell_i$ are $n-1$, and from Lemma \ref{length} (iv) all $n_i$ must be less than or equal to $4$, pairwise distinct, and odd; a contradiction. Thus we have $\ell_1\geq n$, and $n_1=1$. If $\ell_1=n$, we must have $\{\ell_2,\ell_3\}=\{n-1,n-2\}$, $n_i=3$ when $\ell_i=n-1$, and the other $n_j$ in $\{5,7\}$. If $\ell_1=n+1$, we must have $\{\ell_2,\ell_3\}=\{n-1,n-3\}$ or $\ell_2=\ell_3=n-2$. This gives five possible types of solutions up to shift
\begin{equation}
\label{n-13}
\left\{ \begin{array}{ll}
2^{2n-3}(2^{n}-3)+2^{n-2}(3\cdot2^{n-1}-i)+1\cdot(i\cdot2^{n-2}-1),& i\in \{5,7\}\\
2^{2n-3}(2^{n}-i)+2^{n-1}(i\cdot2^{n-2}-3)+1\cdot(3\cdot2^{n-1}-1),& i\in \{5,7\}\\
2^{2n-4}(2^{n+1}-i)+2^{n-1}(i\cdot2^{n-3}-3)+1\cdot(3\cdot2^{n-1}-1),& i\in \{5,7,\cdots,15\}\\
2^{2n-4}(2^{n+1}-3)+2^{n-3}(3\cdot2^{n-1}-i)+1\cdot(i\cdot2^{n-3}-1),& i\in \{5,7,\cdots,15\}\\
2^{2n-4}(2^{n+1}-i)+2^{n-2}(i\cdot2^{n-2}-j)+1\cdot(j\cdot2^{n-2}-1),& i,j\in \{3,5,7\},~ i\neq j\\
\end{array}\right.
\end{equation}

When $w=4$, we show as above that $\ell_1=n+1$, one of the $\ell_i$ is $n-1$, and the other two are equal to $n-2$. The corresponding initial terms are $1$, $3$, $5$ and $7$, and we get the solutions
\begin{equation}
\label{n-14}
\left\{ \begin{array}{l}
2^{3n-5}(2^{n+1}-3)+2^{2n-4}(3\cdot2^{n-1}-i)+2^{n-2}(i\cdot2^{n-2}-j)+1\cdot(j\cdot2^{n-2}-1),~ i,j\in \{5,7\},~ i\neq j\\
2^{3n-5}(2^{n+1}-i)+2^{2n-3}(i\cdot2^{n-2}-3)+2^{n-2}(3\cdot2^{n-1}-j)+1\cdot(j\cdot2^{n-2}-1),~ i,j\in \{5,7\},~ i\neq j\\
2^{3n-5}(2^{n+1}-i)+2^{2n-3}(i\cdot2^{n-2}-j)+2^{n-1}(j\cdot2^{n-2}-3)+1\cdot(3\cdot2^{n-1}-1),~ i,j\in \{5,7\},~ i\neq j\\
\end{array}\right.
\end{equation}

Finally, the case $w=5$ is impossible; as above we must have $\ell_1\leq n+1$, and all other $\ell_i\leq n-1$, with equality for at most one. But we cannot get  $\sum \ell_i=5n-5$ in this way.

We summarize these results in the following

\begin{proposition}
\label{dens22}
Assume $p=2$ and $D=\{1\leq i\leq 2^{n+1}-3,~i\equiv 1\mod 2\}$, with $n$ large enough.

\begin{itemize}
\item[(i)] The minimal irreducible solutions have density $\frac{1}{n}$, they are the
$$1\cdot(2^n-1),~ 2^{n-1}(2^{2n+1}-3)+1\cdot(3\cdot2^n-1)$$

\item[(ii)] There does not exist any irreducible solution having density in the interval
 $\left]\frac{1}{n},\frac{2}{2n-1}\right[$, and the irreducible solutions having density $\frac{2}{2n-1}$ are, up to shift 
$$\left\{
\begin{array}{rl}
2^{n-1}\cdot(2^{n}-3)+1\cdot(3\cdot 2^{n-1}-1)=2^{2n-1}-1 & \\

2^{n-2}\cdot(2^{n+1}-i)+1\cdot(2^{n-2}i-1)=2^{2n-1}-1,& i\in \{3,5,7\}\\

\end{array}
\right.$$

\item[(iii)] There are exactly four (up to shift) irreducible solutions having density in the interval
 $\left]\frac{2}{2n-1},\frac{1}{n-1}\right[$, and all have density $\frac{3}{3n-2}$. These are the following ones, where $i\in \{5,7\}$
$$\left\{
\begin{array}{l}
2^{2n-3}\cdot(2^{n+1}-3)+2^{n-2}\cdot(3\cdot2^{n-1}-i)+1\cdot(i\cdot 2^{n-2}-1)=2^{3n-2}-1 \\
2^{2n-3}\cdot(2^{n+1}-i)+2^{n-1}\cdot(i\cdot2^{n-2}-3)+1\cdot(3\cdot 2^{n-1}-1)=2^{3n-2}-1 \\
\end{array}
\right.$$

\item[(iv)] The solutions having density $\frac{1}{n-1}$ are the ones given in (\ref{n-11}) to (\ref{n-14}).
\end{itemize}

\end{proposition}

From the above results, we deduce the $2$-densities of certain sets of exponents in the following result. Note that from \cite{dens}, this gives the first slopes of the generic Newton polygons $\GNP(D,2)$ for the sets $D$ under consideration.

\begin{corollary}
Assume $d$ is an odd integer, large enough, and let $n$ denote the integer such that $2^n-1\leq d\leq 2^{n+1}-3$.

\begin{itemize}

\item[(i)] Assume $2^n-1< d< 3\cdot 2^{n-1}-1$, then the $2$-density of the set $D:=\{1\leq i\leq d,~ (2,i)=1\}\backslash\{2^n-1\}$ is $\frac{1}{n-1}$.

\item[(ii)] Assume $d=3\cdot 2^{n-1}-1$,
\begin{itemize}

\item[(a)] the $2$-density of the set $D:=\{1\leq i\leq d,~ (2,i)=1\}\backslash\{2^n-1\}$ is $\frac{2}{2n-1}$.
\item[(b)] the $2$-density of the set $D:=\{1\leq i\leq d,~ (2,i)=1\}\backslash\{2^n-3,2^n-1\}$ is $\frac{1}{n-1}$.
\end{itemize}
\item[(ii)] Assume $3\cdot 2^{n-1}-1<d<2^{n+1}-7$,
\begin{itemize}

\item[(a)] the $2$-density of the set $D:=\{1\leq i\leq d,~ (2,i)=1\}\backslash\{2^n-1\}$ is $\frac{2}{2n-1}$.
\item[(b)] the $2$-density of any of the sets $D:=\{1\leq i\leq d,~ (2,i)=1\}\backslash\{2^n-3,2^n-1\}$ and $D:=\{1\leq i\leq d,~ (2,i)=1\}\backslash\{2^n-1,3\cdot 2^{n-1}-1\}$ is $\frac{1}{n-1}$.
\end{itemize}

\item[(iv)] Assume $2^{n+1}-7\leq d\leq 2^{n+1}-5$, then the $2$-density of the set $D:=\{1\leq i\leq d,~ (2,i)=1\}\backslash\{2^n-1\}$ is $\frac{2}{2n-1}$.

\item[(v)] Assume $ d=2^{n+1}-3$, then the $2$-density of the set $D:=\{1\leq i\leq d,~ (2,i)=1\}\backslash\{2^n-1,3\cdot 2^{n-1}-1\}$ is $\frac{2}{2n-1}$.
\end{itemize}

\end{corollary}

\section{First vertices}
\label{s3}

Let $C$ denote an hyperelliptic curve of genus $g$, and $2$-rank $0$, defined over the finite field $\F_q$; from \cite[Proposition 4.1]{sz1}, such a curve admits an equation of the form
$$C_f:~ y^2+y=f(x):=\sum_{i=0}^g c_{2i+1}x^{2i+1},~c_{2i+1}\in \F_q,~c_{2g+1}\neq 0$$
Our aim here is to give the first vertex for the Newton polygon $\NP(C_f)$ of the numerator of the zeta function of this curve for as many polynomials $f$ as possible. 

To obtain this, we give a congruence for the numerator $L(C_f,T)$. We know that $L(C_f,T)=L(f;T)$, and we can apply the congruence in \cite[Remark 1]{cong}. For any polynomial $f$ having its exponents in $D$, this last congruence can be written 

$$L(C_f,T)\equiv \det\left(\I-M(\Gamma)^{\tau^{m-1}}\cdots M(\Gamma)\pi^{m(p-1)\delta}T\right)  \mod I_{\delta}$$

where the matrix $M(\Gamma)$ is defined in \cite[Definition 3.6]{cong} from the minimal irreducible solutions associated to $D$ and the prime $2$, and $\delta$ is the $2$-density of this set. 

The reduction modulo $p$ of the matrix $M(\Gamma)$, $\overline{M}(\Gamma)$, has a rather simple description in the cases under consideration here. If $U_1,\cdots,U_k$ are the minimal irreducible solutions up to shift, the union of their supports is a set $\Sigma:=\{s_1,\ldots,s_N\}$ of positive integers, the \emph{minimal support}. Then $\overline{M}(\Gamma)$ is the $N\times N$ matrix whose $(i,j)$ coefficient is 
$$m_{ij}=\left\{
\begin{array}{ll}
 1 & \textrm{ if } s_j=2s_i \\
c_d & \textrm{ when } 2s_i-s_j=d \textrm{ and we have } u_{dr}=1 \textrm{ for some } r \textrm{ and some minimal irreducible solution } U \\
0 & \textrm{ else }\\
\end{array} \right.$$ 

In the following, we denote by $e(s_1),\ldots,e(s_N)$ the canonical basis of $k^N$, and we denote by $\varphi$ the Frobenius linear morphism of $k^N$ whose matrix in this basis is the transpose of $\overline{M}(\Gamma)$.

Our main tool to determine the first vertex will be the following \cite[Corollary 3.2]{blabla}

\begin{proposition}
\label{firstvertex}
Notations are as above. Denote by $V_{ss}$ the space $\cap_{n\geq 0} \Ima \varphi^n$. Assume $V_{ss}\neq \{0\}$; then the first vertex of the Newton polygons $\NP_q(f)=\NP_q(C_f)$ is $(\dim V_{ss},\delta\dim V_{ss})$.
\end{proposition}

Our strategy is the following: we start with the case $D=\{1\leq i\leq 2g+1,~ (i,2)=1\}$; in this case we get the first vertex for the generic Newton polygon associated to genus $g$ and $2$-rank $0$ hyperelliptic curves. We also determine a polynomial in the coefficients of $f$, the Hasse polynomial $H(f)$, which tells us that the polygon $\NP(f)$ has generic first vertex exactly when $H(f)\neq 0$. This is Theorem \ref{fv}, proven in the first subsection.

In some cases, this polynomial consists of a single monomial. When this happens, we can use the same method to determine the first vertices for most polynomials satisfying $H(f)=0$; actually we just have to replace the set $D$ by the sets $D_i:=D\backslash\{i\}$ for all $i$ such that the variable $c_i$ appears in the monomial $H$. This gives Theorem \ref{fv2}, proven in the second subsection.

\subsection{The generic Newton polygon}

We consider the set $D:=\{1,3,\ldots,2g+1\}$; let $n$ be the unique positive integer satisfying the inequalities $2^n-1\leq 2g+1\leq 2^{n+1}-3$. From the above section, the $2$-density of $D$ is always $\frac{1}{n}$, but the minimal solutions and the minimal support depend on whether we have $2g+1=2^{n+1}-3$ or not. Thus we have to consider two possibilities when we compute the minimal irreducible solutions associated to $D$, and the matrix $M(\Gamma)$.

In the following, we consider the matrices in ${\bf M}_n(\F_q)$ defined by
$$A_n(c):=
\left(\begin{array}{cc}
0 & \textbf{I}_{n-1} \\
c & 0 \\
\end{array}\right),~ 
B_n(c):=
\left(\begin{array}{cc}
0 & \textbf{O}_{n-1} \\
c & 0 \\
\end{array}\right)
$$

In terms of these matrices, we have

\begin{itemize}
	\item[(i)] assume $2^n-1\leq 2g+1<2^{n+1}-3$; there is a unique (up to shift) minimal irreducible solution $1\cdot(2^{n}-1)\equiv 0\mod 2^n-1$, the minimal support is $\{1,\ldots,2^{n-1}\}$, and we have
$$\overline{M}(\Gamma)=A_n(c_{2^n-1})$$
This matrix is invertible exactly when $c_{2^n-1}\neq 0$, and in this case the space $V_{ss}$ has dimension $n$. With the help of Proposition \ref{firstvertex}, this proves the first assertion of Theorem \ref{fv}.

	\item[(ii)] when $2g+1=2^{n+1}-3$; the two (up to shift) minimal irreducible solutions are given above, the minimal support is $\{1,\ldots,2^{n},3,\ldots,3\cdot 2^{n-2}\}$, and we have (note that we should have $c_{2g+1}$ at the $(n,n+1)$ place, but we can assume $c_{2g+1}=1$ from \cite[Proposition 4.1]{sz} to get a simpler form)

$$\overline{M}(\Gamma)=\left(\begin{array}{cc}
A_n(c_{2^n-1}) & B_n(1) \\
B_n(c_{3\cdot2^{n-1}-1}) & A_n(0) \\
\end{array}\right)$$
This matrix is invertible exactly when $c_{3\cdot2^{n-1}-1}\neq 0$, and in this case the space $V_{ss}$ has dimension $2n$. With the help of Proposition \ref{firstvertex}, this proves assertion (iia) of Theorem \ref{fv}. When $c_{3\cdot2^{n-1}-1}= 0$, the space $V_{ss}$ is generated by the first $n$ vectors exactly when $c_{2^n-1}\neq 0$, and this is assertion (iib) of Theorem \ref{fv}.
\end{itemize}

\subsection{Beyond the generic case}

We now consider what happens when the coefficients in Theorem \ref{fv} vanish; i.e. when we are in one of the two following cases (note that in the case $2g+1=2^n-1$ the first slope is always $(n,1)$)
\begin{itemize}
\item[(i)] $2^n-1<2g+1<2^{n+1}-3$ and $c_{2^n-1}=0$; this boils down to considering the set $D=\{1\leq i\leq 2g+1,~ (2,i)=1\}\backslash\{2^n-1\}$, and the associated minimal irreducible solutions.
\item[(ii)] $2g+1=2^{n+1}-3$ and $c_{2^n-1}=c_{3\cdot 2^{n-1}-1}=0$; this boils down to considering the set $D=\{1\leq i\leq 2g+1,~ (2,i)=1\}\backslash\{2^n-1,3\cdot 2^{n-1}-1\}$, and the associated minimal irreducible solutions.
\end{itemize}

Actually, we have to consider different cases: the set $D$ increases with the genus, and new solutions appear in Proposition \ref{dens22}, changing the density and the matrix $M(\Gamma)$.

\subsubsection{The case $d=2g+1< 3\cdot 2^{n-1}-1$} 

From Proposition \ref{dens22}, there is no solution having density in $\left]\frac{1}{n},\frac{1}{n-1}\right[$ for the set $D$. We have to consider solutions having density $\frac{1}{n-1}$.

First assume we have $d< 5\cdot2^{n-2}-1$; the only solutions are $1\cdot(2^{n-1}-1)$ with weight $1$, and $2^{n-2}(2^n-3)+1\cdot(3\cdot2^{n-2}-1)$ with weight $2$. We get the minimal support $\{1,\ldots,2^{n-1},3,\ldots,3\cdot2^{n-3}\}$, and we are exactly in the situation described when we were looking for the first vertex of the generic Newton polygon in the case $2g+1=2^{n+1}-3$, with $n$ being replaced by $n-1$. Thus the first vertex is $(2n-2,2)$ when $c_{2^n-3}c_{3\cdot2^{n-2}-1}\neq 0$, and $(n-1,1)$ when $c_{2^n-3}c_{3\cdot2^{n-2}-1}=0$ and $c_{2^{n-1}-1}\neq 0$.

If we have $5\cdot2^{n-2}-1\leq d\leq 3\cdot2^{n-1}-7$, we have a new solution, namely $2^{n-2}(2^n-5)+1\cdot(5\cdot2^{n-2}-1)$. The minimal support is now $\{1,\ldots,2^{n-1},3,\ldots,3\cdot2^{n-3},5,\ldots,5\cdot2^{n-3}\}$. Computing the iterates of the vector $e(1)$ under $\varphi$, we get that $V_{ss}$ has dimension $2(n-1)$ exactly when $c_{2^n-3}^{2^{n-2}}c_{3\cdot2^{n-2}-1}+c_{2^n-5}^{2^{n-2}}c_{5\cdot2^{n-2}-1}\neq 0$; thus the first vertex of $\NP(C_f)$ is $(2(n-1),2)$ if, and only if this polynomial is non zero. If it is zero, the space $V_{ss}$ has dimension $n-1$ if, and only if $c_{2^{n-1}-1}\neq 0$, and the first vertex is $(n-1,1)$ in this case.

When $d= 3\cdot2^{n-1}-5$, we get a new solution, of weight $3$, namely $2^{2n-3}(2^n-3)+2^{n-2}(3\cdot2^{n-1}-5)+1\cdot(5\cdot2^{n-2}-1)$. The minimal support is now $\{1,\ldots,2^{n-1},3,\ldots,3\cdot2^{n-2},5,\ldots,5\cdot2^{n-3}\}$. The matrix $\overline{M}(\Gamma)$ is in ${\bf M}_{3n-3}(\F_q)$ and it has determinant
$$c_{2^n-3}c_{3\cdot2^{n-1}-5}c_{5\cdot2^{n-2}-1}$$
The first vertex is $(3(n-1),3)$ exactly when this determinant is non zero. When it vanishes, we are reduced to the preceding case with the additional assumption $c_{2^n-3}c_{5\cdot2^{n-2}-1}=0$.

When $d= 3\cdot2^{n-1}-3$, we get a new solution, of weight $1$, namely $1\cdot(3\cdot2^{n-1}-3)$; this does not change the minimal support, but adds the new coefficient $c_{3\cdot2^{n-1}-3}$ at the intersection of line $2n-1$ and row $n+1$, and the determinant becomes
$$c_{5\cdot2^{n-2}-1}\left(c_{2^n-3}c_{3\cdot2^{n-1}-5}+c_{2^n-5}c_{3\cdot2^{n-1}-3}\right)$$

\subsubsection{The case $3\cdot 2^{n-1}-1\leq 2g+1< 2^{n+1}-7$} 

From Proposition \ref{dens22}, the $2$-density of the set $\{1,\ldots,2g+1\}\backslash\{2^n-1\}$ is $\frac{2}{2n-1}$, and the unique minimal solution is the first one in assertion (ii) of this Proposition. The minimal support is the support of this solution, namely $\{1,\cdots,2^{n-1},3,\cdots,3\cdot2^{n-2}\}$, and we have $\varphi(e(2^i))=e(2^{i+1})$ for $0\leq i\leq n-2$, $\varphi(e(3\cdot2^i))=e(3\cdot2^{i+1})$ for $0\leq i\leq n-3$, $\varphi(e(2^{n-1}))=c_{2^n-3}e(3)$ and $\varphi(e(3\cdot2^{n-2}))=c_{3\cdot2^{n-1}-1}e(1)$. We see that $V=V^{ss}$ if, and only if we have $c_{2^n-3}c_{3\cdot2^{n-1}-1}\neq 0$ since in this case any of the $e(i)$ is a cyclic vector; else we have $V=V^{nil}$.

In this case we can go one step further when $c_{2^n-3}c_{3\cdot2^{n-1}-1}=0$. This boils down to considering one of the sets $D':=D\backslash\{2^n-3\}$ or $D'':=D\backslash\{3\cdot2^{n-1}-1\}$. From the calculations above, both have density $\frac{1}{n-1}$ since the solutions having density $\frac{3}{3n-2}$ need an element of the form $2^{n+1}-i$, $i\in \{5,7\}$. We conclude that the generic first slope is $\frac{1}{n-1}$. We do not compute the first vertex nor its Hasse polynomial since there are many solutions having density $\frac{1}{n-1}$.

\subsubsection{The case $d=2^{n+1}-7$} 

The density of the set $\{1,\ldots,2g+1\}\backslash\{2^n-1\}$ remains $\frac{2}{2n-1}$, we get the new minimal irreducible solution of Proposition \ref{dens22} (i) with $i=7$, and the new elements $2^n$ and $7,\ldots,7\cdot2^{n-3}$ in the minimal support. The action of $\varphi$ is as described above, except $\varphi(e(2^{n-1}))=c_{2^n-3}e(3)+e(2^n)$, and the new $\varphi(e(2^{n}))=c_{2^{n+1}-7}e(7)$, $\varphi(e(7\cdot2^i))=e(7\cdot2^{i+1})$ for $0\leq i\leq n-4$ and $\varphi(e(7\cdot2^{n-3}))=c_{7\cdot2^{n-2}-1}e(1)$. In this case, the vector space generated by the iterates of $e(1)$ contains $V_{ss}$ (clearly the iterates of any basis vector land finally in this space). Moreover, the $2n-1$ vectors $\varphi^i(e(1))$, $0\leq i\leq 2n-2$ are linearly independent since $c_{2^{n+1}-7}$ is non zero, and we get 
$$\varphi^{2n-1}(e(1))=
\left(c_{2^{n+1}-7}^{2^{n-2}}c_{7\cdot2^{n-2}-1}+
c_{2^{n}-3}^{2^{n-1}}c_{3\cdot2^{n-1}-1}\right)e(1)$$
Thus the vector space $V_{ss}$ has dimension $2n-1$ exactly when the polynomial above is non zero.

\subsubsection{The case $d=2^{n+1}-5$} 

We reason the same way (with the new solution from assertion (i) with $i=5$, giving the new elements $5,\ldots,5\cdot2^{n-3}$ in the minimal support, etc...), and we get
$$\varphi^{2n-1}(e(1))=\left(c_{2^{n+1}-5}^{2^{n-2}}c_{5\cdot2^{n-2}-1}+
c_{2^{n+1}-7}^{2^{n-2}}c_{7\cdot2^{n-2}-1}+
c_{2^{n}-3}^{2^{n-1}}c_{3\cdot2^{n-1}-1}\right)e(1)$$

\subsubsection{The case $d=2^{n+1}-3$} 

Here we have $c_{2^n-1}=c_{3\cdot 2^{n-1}-1}=0$ in order for the first slope to be greater than $\frac{1}{n}$. From Proposition \ref{dens22}, the density of the set $\{1,\ldots,d\}\backslash\{2^n-1,3\cdot 2^{n-1}-1\}$ is $\frac{2}{2n-1}$, with the last three minimal irreducible solutions from Assertion (ii). We deduce the minimal support 
$$\{1,\cdots,2^{n},3,\cdots,3\cdot2^{n-3},5,\cdots,5\cdot2^{n-3},7,\cdots,7\cdot2^{n-3}\}$$
and the action of $\varphi$, given by $\varphi(e(2^i))=e(2^{i+1})$ for $0\leq i\leq n-1$, $\varphi(e(k\cdot2^i))=e(k\cdot2^{i+1})$, $\varphi(e(k\cdot2^{n-3}))=c_{k\cdot2^{n-2}-1}e(1)$ for $0\leq i\leq n-4$, $k\in\{3,5,7\}$, and
$\varphi(e(2^{n}))=c_{2^{n+1}-3}f(e(3))+c_{2^{n+1}-5}f(e(5))+c_{2^{n+1}-7}f(e(7))$.

Once again, we consider the (cyclic) subspace of $V$ generated by $e(1)$ and its iterates; it is clear from the description of $\varphi$ that the iterates of any of the vectors of the basis fall into this space; thus we have $V_1\supset V^{ss}$. When we compute the iterates of $e(1)$, we find the vectors $e(1),\ldots,\varphi^{2n-2}(e(1))$ are linearly independent since $c_{2^{n+1}-3}$ is non zero, and from the relation  
$$\varphi^{2n-1}(e(1))=\left(c_{2^{n+1}-3}^{2^{n-2}}c_{3\cdot2^{n-2}-1}+c_{2^{n+1}-5}^{2^{n-2}}
c_{5\cdot2^{n-2}-1}+c_{2^{n+1}-7}^{2^{n-2}}c_{7\cdot2^{n-2}-1}\right)e(1)$$
we deduce that $\dim V^{ss}=2n-1$ if, and only if the above polynomial is non zero.


\begin{thebibliography}{99}


\bibitem{dens} 
R. Blache, \emph{Valuation of exponential sums and the generic first slope for Artin-Schreier curves}, J. Number Theory {\bf 132} (2012), 2336-2352.

\bibitem{cong} 
R. Blache, \emph{Congruences for $L$-functions of additive exponential sums}, preprint arXiv:1206.1387 (2012).

\bibitem{blabla} 
R. Blache, \emph{Valuations of exponential sums and Artin-Schreier curves}, preprint arXiv:1502.00969 (2015).

\bibitem{vv}
G. van der Geer, M. van der Vlugt, \emph{Reed-Muller codes and supersingular curves}, Comp. Math. {\bf 84} (1992), 333-367.

\bibitem{vv2}
G. van der Geer, M. van der Vlugt, \emph{On the existence of supersingular curves of given genus}, J. reine angew. Math. {\bf 458} (1995), 53-61.

\bibitem{mkcs} 
O. Moreno, K.W. Shum, F.N. Castro, P.V. Kumar, \emph{Tight bounds for Chevalley-Warning-Ax-Katz type estimates, with improved applications}, Proc. Lond. Math. Soc. {\bf 88} (2004), 545--564.


\bibitem{oo} 
F. Oort, \emph{Abelian varieties isogenous to a Jacobian}, in Problems from the Workshop on automorphisms of curves, Rend. Mat. Sem. U. Padova {\bf 113} (2005), 129-177.

\bibitem{sz} 
J. Scholten, H.J. Zhu, \emph{Hyperelliptic curves in characteristic $2$}, IMRN {\bf 17} (2002), 905-917.

\bibitem{sz1} 
J. Scholten, H.J. Zhu, \emph{Families of supersingular curves in characteristic $2$}, Math. Res. Let. {\bf 9} (2002), 639-650.


\end{thebibliography}
\end{document}